\documentclass[13pt]{article}  % list options between brackets
\usepackage{amssymb}              % list packages between braces
\usepackage{amsthm}
\usepackage{amsmath}
\usepackage{eucal}
\usepackage{verbatim}
\usepackage[dvips]{graphicx}
\usepackage{multirow}
\usepackage{fancyhdr}
\usepackage{color}
\usepackage{enumerate}
\usepackage{calrsfs}
\usepackage{fullpage}
\usepackage{amsmath}
\usepackage{eufrak}

\newtheorem{theorem}{Theorem}
\newtheorem{lemma}{Lemma}

\newtheorem{corollary}{Corollary}

\makeatletter
\newcommand{\leqnomode}{\tagsleft@true}
\newcommand{\reqnomode}{\tagsleft@false}
\makeatother
\def\({\begin{eqnarray}}
\def\){\end{eqnarray}}
\def\[{\begin{eqnarray*}}
\def\]{\end{eqnarray*}}
\def\part#1#2{\frac{\partial #1}{\partial #2}}

\def\R{\mathbb{R}}
\def\N{\mathbb{N}}
\def\d{\mathrm{d}}
\def\tot#1#2{\frac{\d #1}{\d #2}}
\def\eps{\varepsilon}

\def\grad{\nabla}

\def\taumax{\overline\tau} %{\tau_\mathrm{max}}
\def\P{\mathbb{P}}
\def\W{\mathcal{W}}

\begin{document}

\title{A simple proof of asymptotic consensus in the Hegselmann-Krause
and Cucker-Smale models with renormalization and delay}   % type title between braces
\author{Jan Haskovec\footnote{Computer, Electrical and Mathematical Sciences \& Engineering, King Abdullah University of Science and Technology, 23955 Thuwal, KSA.
jan.haskovec@kaust.edu.sa}}

\date{}

\maketitle

\begin{abstract}
We present a simple proof of asymptotic consensus in the discrete Hegselmann-Krause model
and flocking in the discrete Cucker-Smale model with renormalization and variable delay.
It is based on convexity of the renormalized communication weights
and a Gronwall-Halanay-type inequality.
The main advantage of our method, compared to previous approaches to the delay Hegselmann-Krause model,
is that it does not require any restriction on the maximal time delay, or the initial data,
or decay rate of the influence function. From this point of view the result is optimal.
For the Cucker-Smale model it provides an analogous result in the regime of unconditonal flocking
with sufficiently slowly decaying communication rate, but still without any restriction on the length
of the maximal time delay.
Moreover, we demonstrate that the method can be easily extended to the mean-field
limits of both the Hegselmann-Krause and Cucker-Smale systems,
using appropriate stability results on the measure-valued solutions.
\end{abstract}
\vspace{2mm}

\vspace{2mm}

%SIAM Journal on Control and Optimization
%SIAM Journal on Applied Dynamical Systems

%\textbf{2010 MR Subject Classification}: %34K05, 82C22, 34D05, 92D50.
\vspace{2mm}

%%%%%%%%%%%%%%%%%%%%%%%%%%%%%%%%%%%%%%%%
\section{Introduction}\label{sec:Intro}
In this paper we study asymptotic behavior of the Hegselmann-Krause \cite{HK}
and Cucker-Smale \cite{CS1, CS2} models with normalized communication weights
and variable time delay.
The Hegselmann-Krause model describes the evolution of $N\in\N$ agents who adapt their opinions to the ones of other members of the group.
Agent $i$'s opinion is represented by the quantity $x_i = x_i(t)\in\R^d$, with $d\in\N$ the space dimension, which is a function of time $t\geq 0$.
Assuming that the agents' communication takes place subject to a variable delay $\tau=\tau(t)$,
the opinions evolve according to the following system
\( \label{eq:HK}
   \dot x_i(t) = %\lambda
     \sum_{j=1}^N \psi_{ij}(t) (x_j(t-\tau(t)) - x_i(t)), \qquad i=1,\ldots,N.
\)
%Here $\lambda>0$ is a parameter and
The normalized communication weights $\psi_{ij}=\psi_{ij}(t)$ are given by
\(\label{psi}
   \psi_{ij}(t) := \left\{ \begin{array}{ll}
         \displaystyle \frac{\psi(|x_j(t - \tau(t)) - x_i(t)|)}{\sum_{\ell\neq i} \psi(|x_\ell(t - \tau(t)) - x_i(t)|)} & \textrm{if $j \neq i$,}\\[4mm]
         0 & \textrm{if $j=i$,}
  \end{array} \right.
\)
where the nonnegative \emph{influence function} $\psi:[0,\infty)\to [0,\infty)$,
also called \emph{communication rate},
measures how strongly each agent is influenced by others depending on their distance.
The variable time delay $\tau=\tau(t)$ is assumed to be nonegative and uniformly bounded by some $\taumax>0$;
this includes the generic case of the constant delay $\tau(t)\equiv\taumax$.
The system \eqref{eq:HK} is equipped with the initial datum
\(\label{IC:HK}
   x_i(s) = x^0_i(s),\qquad i=1,\cdots,N, \quad s \in [-\taumax,0],
\)
with prescribed trajectories $x^0_i\in C([-\taumax,0])$, $ i=1,\cdots,N$.

The phenomenon of consensus finding in the context of \eqref{eq:HK} refers
to the (asymptotic) emergence of one or more {opinion clusters} formed
by agents with (almost) identical opinions \cite{JM}. \emph{Global consensus}
is the state where all agents have the same opinion, i.e.,
$x_i=x_j$ for all $i,j \in\{1,\dots,N\}$.
Convergence to global consensus as $t\to\infty$ for the system \eqref{eq:HK} has been
proved in \cite{CPP} under a set of conditions requiring smallness of the maximal time delay
in relation to the decay speed of the influence function and the fluctuation of the initial datum.
However, as we shall discuss in Section \ref{sec:mam}, there is justified expectation that asymptotic global consensus
is reached without any restriction on the maximal time delay $\taumax$, for all initial data,
and only assuming that the influence function
is strictly positive on $[0,\infty)$, but may decay to zero arbitrarily fast at infinity.
In this paper we prove that this intuitive expectation is indeed true.
As this result is essentially optimal, it closes an important gap in the theory
of asymptotic behavior of the Hegselmann-Krause model with delay.

Our method of proof consists of three ingredients.
First, we prove nonexpansivity of the agent group, i.e., a uniform bound on the position radius of the agents.
Second, we establish a convexity argument, exploiting the
following property of the renormalized weights \eqref{psi},
\(  \label{psi:conv}
   \sum_{j=1}^N \psi_{ij} = 1 \qquad\mbox{for all } i=1,\cdots, N.
\)
This implies that the terms $\sum_{j=1}^N \psi_{ij}(t) x_j(t-\tau(t))$ in \eqref{eq:HK} are convex combinations
of the vectors $x_j(t-\tau(t))$.
Finally, a Gronwall-Halanay-type inequality shall
provide exponential decay to zero of solutions of a certain delay differential inequality.

%Various aspects of the consensus behavior of various modifications of \eqref{eq:HK}
%have been studied in, e.g., \cite{Bha, Blondel, Canuto, Carro, Mohajer, Moreau, MT, JM, Wang, Wedin}.

A second goal of this paper is to study, by a slight extension of the above method, the asymptotic behavior of a variant of the
Cucker-Smale model \cite{CS1, CS2} with renormalized communication weights, introduced by Motsch and Tadmor \cite{MT}.
The model can be seen as a second-order version of \eqref{eq:HK}
and reads, for $i=1,\cdots,N$,
\begin{align}\label{eq:CS}
\begin{aligned}
   %\tot{x_i(t)}{t}
      \dot x_i(t) &= v_i(t), \\ %\qquad i=1,\cdots,N, \quad t >0,\\
   %\tot{v_i(t)}{t}
      \dot v_i(t) & = \sum_{j=1}^N \psi_{ij}(t) (v_j(t - \tau(t)) - v_i(t)),
\end{aligned}
\end{align}
%with agent positions $x_i=x_i(t)$ and velocities $v_i=v_i(t)$.
where the normalized communication weights $\psi_{ij}=\psi_{ij}(t)$ are again given by \eqref{psi}.
The system is equipped with the initial datum
\(\label{IC:CS}
   x_i(s) = x^0_i(s),\quad v_i(s) = v^0_i(s), \qquad i=1,\cdots,N, \quad s \in [-\taumax,0],
\)
with prescribed phase-space trajectories $x^0_i, v^0_i\in C([-\taumax,0]; \R^d)$, $ i=1,\cdots,N$.
For physical reasons it may be required that
$$  x_i^0(s) = x^0_i(-\taumax) + \int_{-\taumax}^s v_i^0(\sigma) \d\sigma\quad\mbox{for } s \in (-\taumax,0],$$
but we do not pose this particular restriction here.

The Cucker-Smale model (and its variants) is in the literature considered
a generic model for flocking, herding or schooling of animals
with positions $x_i=x_i(t)$ and velocities $v_i=v_i(t)$.
If the influence function has a heavy tail, i.e., $\int^\infty \psi(s) = \infty$,
then the model exhibits the so-called  \emph{unconditional flocking}, where \eqref{flocking} holds for every initial configuration, see \cite{CS1, CS2, Tadmor-Ha, Ha-Liu}.
In the opposite case the flocking is \emph{conditional},
i.e., the asymptotic behavior of the system depends on %the value of $\lambda$ and on
the initial configuration.
The Cucker-Smale model with delay was studied in \cite{Liu-Wu, EHS, HasMar, ChoiH1, ChoiH2}.
In particular, in \cite{ChoiH1} the precise form \eqref{eq:CS} with the normalized communication rates \eqref{psi} was considered
(with constant time delay) and asymptotic flocking was proved under a smallness condition on the delay,
related to the decay properties of the influence function $\psi$ and the velocity diameter of the initial datum.
In this paper we significantly improve this result by considering much less restrictive conditions
that are close to optimal. In particular, we prove that if the influence function $\psi=\psi(s)$
decays sufficiently slowly at infinity, then asymptotic flocking takes place for all initial data
and without restriction on the delay length. This can be seen as an analogy of the unconditional flocking
result for the original Cucker-Smale model.

The paper is organized as follows. In Section \ref{sec:mam} we explain the intuitive motivation
for our results, formulate the precise assumptions on the influence function $\psi=\psi(s)$ and delay $\tau=\tau(t)$,
and state our main results. Their proofs for the Hegselmann-Krause model \eqref{eq:HK} are presented in
Section \ref{sec:HK}, and their extension for the Cucker-Smale model \eqref{eq:CS} in Section \ref{sec:CS}.
Finally, in Section \ref{sec:MF} we prove the analogues of the consensus and flocking results
for the mean-field limits of \eqref{eq:HK} and \eqref{eq:CS}.

%%%%%%%%%%%%%%%%%%%%%%%%%%%%%%%%%%%%%
\section{Motivation, assumptions and main results}\label{sec:mam}
Let us explain the intuitive expectation that {all} solutions of \eqref{eq:HK}
should converge to global consensus as $t\to\infty$, regardless of the length of the delay.
%or decay properties of the influence function $\psi$.
For this sake, we consider the case of two agents, $N=2$, with positions $x_1(t)$, $x_2(t)$,
and constant delay $\tau(t)\equiv \tau$.
Then \eqref{psi} gives $\psi_{12} = \psi_{21} = 1$ and \eqref{eq:HK} reduces to the linear system
\begin{align*} %\label{HK2part}
\begin{aligned}
   %\tot{}{t}
   \dot x_1(t) &= x_2(t - \tau) - x_1(t), \\
   %\tot{}{t}
   \dot x_2(t) &= x_1(t - \tau) - x_2(t).
\end{aligned}
\end{align*}
Defining %$u:=x_1+x_2$,
$w:=x_1-x_2$, we have
\(
  %\tot{}{t}
%  \dot u(t) &=& u(t - \tau) - u(t), \label{uEq} \\
  %\tot{}{t}
  \dot w(t) = -w(t - \tau) - w(t). \label{wEq}
\)
%The first equation admits, for a constant initial datum, a constant solution,
%which corresponds to conservation of the centre of gravity of the two-agent system.
Assuming a solution of the form $w(t) = e^{\xi t}$ for some complex $\xi\in\mathbb{C}$,
we obtain the characteristic equation
\[
   \xi = -e^{-\xi\tau} - 1.
\]
A simple inspection reveals that all roots $\xi$ have negative real part,
which implies that all solutions $w(t)$ of \eqref{wEq} tend to zero as $t\to\infty$.
I.e., we have the asymptotic consensus $\lim_{t\to\infty} x_1(t)-x_2(t) = 0$,
for any value of the constant delay $\tau$.

Let us note that the situation is fundamentally different for 
a reaction-type delay in \eqref{eq:HK}, i.e., the system
\(   \label{eq:HKr}
   \dot x_i(t) = \sum_{j=1}^N \psi_{ij}(t) (x_j(t-\tau(t)) - x_i(t-\tau(t))), \qquad i=1,\ldots,N.
\)
Then, considering again $N=2$ and constant delay $\tau(t)\equiv\tau$, we have for $w:=x_1-x_2$,
\[
    \dot w(t) = -2 w(t - \tau).
\]
Nontrivial solutions of this equation exhibit oscillations whenever $2\tau > e^{-1}$
and the amplitude of the oscillations diverges in time if $2\tau > \pi/2$, see, e.g., \cite{Smith}.
The different types of asymptotic behavior of the Hegselmann-Krause system with communication-type delay \eqref{eq:HK}
versus the system with reaction-type delay \eqref{eq:HKr} can be intuitively understood by noting
that the instantaneous negative feedback term $-x_i(t)$ has a stabilizing effect,
while the delay terms typically destabilize, and this effect becomes stronger with longer delays.
Thus, the stabilizing effect of the instantaneous term $-x_i(t)$ in the right-hand side of \eqref{eq:HK}
is stronger than the destabilizing effect of the delay terms $x_j(t-\tau(t))$, regardless of the maximal length of the delay $\taumax$.
In \eqref{eq:HKr} the stabilizing effect is not present, and therefore, for large enough delays, the solutions
may diverge as $t\to\infty$. % not reach asymptotic consensus.

 Let us now formulate the precise assumptions on the influence function and variable delay
 that we adopt throughout the paper.
The influence function $\psi\in C([0,\infty))$ shall satisfy
\( \label{ass:psi}
   0 < \psi(s) \leq 1\qquad \mbox{for all } s\geq 0.
\)
Note that we do not require any monotonicity properties of $\psi$.
%although this will be indeed assumed when we shall study the Cucker-Smale system in Section \ref{sec:CS}.
For the variable delay function $\tau\in C([0,\infty))$ we pose the assumption
\( \label{ass:tau}
   0 \leq \tau(t) \leq \taumax \qquad \mbox{for all } t\geq 0,
\)
for some fixed $\taumax>0$.
Clearly, the generic case of constant delay $\tau(t)\equiv\taumax$ is included in \eqref{ass:tau}.
Let us stress that both the above assumptions are very minimal.
%compared with 
The global positivity of $\psi$ in \eqref{ass:psi} cannot be further relaxed,
since universal consensus behavior cannot be expected if $\psi$
would be allowed to vanish even pointwise. The upper bound $\psi\leq 1$
can be replaced by any arbitrary positive value due to the scaling invariance of \eqref{psi}.
The variable delay $\tau=\tau(t)$ is allowed to vanish on arbitrary subsets of $[0,\infty)$.

To formulate our main results, let us introduce the spatial diameter $d_x=d_x(t)$ of the agent group,
\(  \label{def:dX}
   d_x(t) := \max_{1 \leq i,j \leq N}|x_i(t) - x_j(t)|.
\)
The \emph{global asymptotic consensus} is then defined as the property
\( \label{def:consensus}
    \lim_{t\to\infty} d_x(t) = 0.
\)
We note that, in general, \eqref{eq:HK} does not conserve the mean value
$\frac{1}{N}\sum_{i=1}^N x_i$. Consequently, the (asymptotic) consensus vector cannot be
inferred from the initial datum in a straightforward way and can be seen as an emergent property of the system.

Our main result regarding the consensus behavior of the Hegselmann-Krause model \eqref{eq:HK} is as follows.

\begin{theorem}\label{thm:HK}
Let $N\geq 3$ and let the assumptions \eqref{ass:psi} on $\psi=\psi(s)$ and \eqref{ass:tau} on $\tau=\tau(t)$ be verified.
Then all solutions of \eqref{eq:HK} reach global asymptotic consensus as defined by \eqref{def:consensus}.
The decay of $d_x=d_x(t)$ to zero is exponential with calculable rate that improves with increasing $N$.
%\[
 %  \lim_{t\to\infty} d_x(t) = 0.
%\]
\end{theorem}

For the Cucker-Smale-type system \eqref{eq:CS}, we define \emph{asymptotic flocking} as the property
\( \label{flocking}
   \lim_{t\to\infty} d_v(t) = 0, \qquad \sup_{t\geq 0} d_x(t) < \infty,
\)
where the velocity diameter $d_v=d_v(t)$ of the agent group is given by
\[  %  \label{def:dV}
   d_v(t) := \max_{1 \leq i,j \leq N}|v_i(t) - v_j(t)|.
\]
We also introduce the velocity radius $R_v=R_v(t)$,
\[ % \label{def:RV}
   R_v(t) := \max_{1 \leq i \leq N} |v_i(t)|.
\]
Again, we observe that \eqref{eq:CS} does not conserve the global momentum $\frac{1}{N}\sum_{i=1}^N v_i$.
Finally, we introduce the quantity
\(  \label{def:Psi}
   \Psi(r) := \min_{s\in [0,r]} \psi(s).
\)

\begin{theorem}\label{thm:CS}
Let $N\geq 3$ and let the assumptions \eqref{ass:psi} on $\psi=\psi(s)$ and \eqref{ass:tau} on $\tau=\tau(t)$ be verified.
Moreover, assume that there exists $C\in (0,1)$ such that
\(  \label{ass:C}
   1-C = \left( 1- \frac{N-2}{N-1}  \Psi\left(\taumax R_v^0 + d_x^0 + \frac{d_v^0}{C} \right)\right) e^{C\taumax},
\)
with
\(  \label{Rdd0}
   R_v^0 := \max_{s\in [-\taumax,0]} R_v(s), \qquad
   d_x^0 := \max_{s\in [-\taumax,0]} d_x(s), \qquad
   d_v^0 := \max_{s\in [-\taumax,0]} d_v(s).
\)
Then the solution $(x(t),v(t))$ of \eqref{eq:CS} subject to the initial datum \eqref{IC:CS}
exhibits asymptotic flocking as defined by \eqref{flocking}.
Moreover, the decay of the velocity diameter $d_v=d_v(t)$ to zero as $t\to\infty$
is exponential with rate $C$.
\end{theorem}

Let us note that the assumption \eqref{ass:C}, although it may seem technical, is in fact very natural
in the context of the Cucker-Smale model. Indeed, for the influence function
\(  \label{psi:CS}
   \psi(s) = \frac{1}{(1+s^2)^\beta} \qquad\mbox{for } s\geq 0,
\)
which was considered in the original works \cite{CS1, CS2},
assumption \eqref{ass:C} is satisfied whenever $\beta < 1/2$,
regardless of the particular values of $R_v^0$, $d_x^0$ and $d_v^0$.
This is precisely the setting that leads to \emph{unconditional flocking} in the original
Cucker-Smale model, and $\beta>1/2$ asymptotic flocking may or may not take place,
depending on the initial datum. From this point of view, our result is very close to optimal.
We formulate it as the following corollary and prove it in Section \ref{sec:CS}.

\begin{corollary}\label{corr:CS}
Consider the influence function given by \eqref{psi:CS} with $\beta<1/2$.
Then all solutions of \eqref{eq:CS} %with the influence function \eqref{psi:CS}
exhibit asymptotic flocking as defined by \eqref{flocking}.
\end{corollary}

The last part of the paper is devoted to the proof of asymptotic consensus and flocking
in the mean-field limits of the Hegselmann-Krause model \eqref{eq:HK}
and the Cucker-Smale system \eqref{eq:CS}.
Letting $N\to\infty$ in \eqref{eq:HK} leads to the conservation law
\(  \label{mf:HK}
   \partial_t f + \grad_x\cdot (F[f] f) = 0,
\)
for the time-dependent probability measure $f=f(t,x)$
which describes the probability of finding an agent at time $t\geq 0$
located at $x\in\R^d$.
The operator $F=F[f]$ is defined as
\(  \label{F}
   F[f](t,x) := \frac{\int_{\R^d} \psi(|x-y|) (y-x) f(t-\tau(t),y) \d y}{\int_{\R^d} \psi(|x-y|) f(t-\tau(t),y) \d y}.
\)
The system \eqref{mf:HK}--\eqref{F} is equipped with the initial datum
$f(t) = f^0(t)$, $t\in [-\taumax,0]$, with $f^0 \in C([-\taumax,0], \P(\R^d))$,
where $\P(\R^d)$ denotes the set of probability measures on $\R^d$.
We shall assume that the initial datum is uniformly compactly supported,
i.e., %that there exists a compact set $K\subset \R^d$ such that
\(   \label{IC:mf:HK}
%   \mathrm{supp}\, f^0(t) \subset K \qquad\mbox{for all } t\in[-\taumax,0].
   \sup_{s\in [-\taumax,0]} d_x[f^0(s)] < \infty,
\)
where the diameter $d_x[h]$ for a probability measure $h\in \P(\R^d)$ is defined as
\[
   d_x[h] := \sup \{ |x-y|,\, x,y\in \mathrm{supp}\, h \}.
\]
In Section \ref{sec:MF} we shall prove the following theorem,
which is a direct consequence of a stability estimate in terms of the
Monge-Kantorowich-Rubinstein distance, combined with the fact
that the consensus estimates derived in Section \ref{sec:HK} are uniform with respect
to the number of agents $N\in\N$.

\begin{theorem}\label{thm:mf:HK}
Let the assumptions \eqref{ass:psi} on $\psi=\psi(s)$ and \eqref{ass:tau} on $\tau=\tau(t)$ be verified.
Then all solutions $f=f(t)$ of \eqref{mf:HK} with compactly supported initial datum \eqref{IC:mf:HK}
reach global asymptotic consensus in the sense
\[  %  \label{mf:HK:cons}
   \lim_{t\to\infty} d_x[f(t)] = 0,
\]
and the decay is exponential.
\end{theorem}

The mean-field limit of the Cucker-Smale system \eqref{eq:CS} is given,
in its strong formulation, by the kinetic equation
\(  \label{mf:CS}
   \partial_t g + v\cdot\grad_x g + \grad_v\cdot (G[g] g) = 0,
\)
with
\(   \label{G}
   G[g](t,x) := \frac{\int_{\R^d}\int_{\R^d} \psi(|x-y|) (w-v) f(t-\tau(t),y,w) \d y\d w}{\int_{\R^d}\int_{\R^d} \psi(|x-y|) f(t-\tau(t),y,w) \d y\d w}.
\)
Here the time-dependent probability measure $g=g(t,x,v)$ describes the probability of finding an agent at time $t\geq 0$
located at $x\in\R^d$ with velocity $v\in\R^d$.
The initial datum $g(t) = g^0(t)$, $t\in [-\taumax,0]$, with $g^0 \in C([-\tau,0], \P(\R^d\times\R^d))$,
is again assumed to be uniformly compactly supported in the sense that there exists a compact set $K\subset \R^d\times\R^d$ such that
\(   \label{IC:mf:CS}
      \mathrm{supp}\, g^0(t) \subset K \qquad\mbox{for all } t\in[-\taumax,0].
\)
With a slight abuse of notation, define the position and velocity diameters of $\mathrm{supp}\, g(t)$ as
\[  %   \label{mf:diameters}
   %\mathrm{diam}_x (\mathrm{supp}\,  g)
   d_x[g(t)] := \sup \{ |x-y|,\, x,y\in \mathrm{supp}_x\, g(t) \}, \qquad
   %\mathrm{diam}_v (\mathrm{supp}\,  g)
   d_v[g(t)]:= \sup \{ |v-w|,\, v,w\in \mathrm{supp}_v\, g(t) \},
\]
where $ \mathrm{supp}_x\, g$ and, resp., $ \mathrm{supp}_v\, g$, are projections
of $ \mathrm{supp}\, g$ onto the $x$- and, resp., $v$-variables.
The velocity radius is defined as
\[
   R_v[g(t)] := \sup \{ |w|,\, w\in \mathrm{supp}_v\, g(t) \}.
\]
We then have the following flocking result.

\begin{theorem}\label{thm:mf:CS}
Let the assumptions \eqref{ass:psi} on $\psi=\psi(s)$ and \eqref{ass:tau} on $\tau=\tau(t)$ be verified.
Moreover, assume that there exists $C\in (0,1)$ such that
\(  \label{ass:mf:C}
   1-C = \left( 1- \Psi\left(\taumax R_v^0 + d_x^0 + \frac{d_v^0}{C} \right)\right) e^{C\taumax},
\)
with $\Psi$ defined in \eqref{def:Psi} and
\(  \label{Rdd0}
   R_v^0 := \max_{s\in [-\taumax,0]} R_v[g^0(s)], \qquad
   d_x^0 := \max_{s\in [-\taumax,0]} d_x[g^0(s)], \qquad
   d_v^0 := \max_{s\in [-\taumax,0]} d_v[g^0(s)].
\)
Then the solution of \eqref{mf:CS} subject to the compactly supported initial datum \eqref{IC:mf:CS}
exhibits asymptotic flocking in the sense
\[
   \lim_{t\to\infty} d_v[g(t)] = 0, \qquad \sup_{t\geq 0} d_x[g(t)] < \infty.
\]
Moreover, the decay of the velocity diameter to zero as $t\to\infty$ is exponential with rate $C$.
\end{theorem}

Again, we note that if the influence function is of the form $\psi(s) = (1+s^2)^{-\beta}$,
assumption \eqref{ass:mf:C} is satisfied whenever $\beta<1/2$, regardless of
the values of $\taumax$, $R_v^0$, $d_x^0$ and $d_v^0$.

%%%%%%%%%%%%%%%%%%%%%%%%%%%%%%%%%%
\section{Asymptotic consensus for the Hegselmann-Krause model \eqref{eq:HK}}\label{sec:HK}
We define the radius of the group $R_x=R_x(t)$ by
\(  \label{def:diam}
   R_x(t) := \max_{1\leq i\leq N} |x_i(t)|.
\)
The following lemma shows that the radius is bounded uniformly in time by the radius of the initial datum,
defined as
\(  \label{def:R0}
   R_x^0 := \max_{t\in [-\taumax,0]} R_x(t).
\)

\begin{lemma}\label{lem:Rxbound}
Let the initial datum $x^0\in C([-\taumax,0];\R^d)^N$
and let $R_x^0$ be given by \eqref{def:R0}.
Then, along the solutions of \eqref{eq:HK}--\eqref{IC:HK}, the diameter $R_x=R_x(t)$ defined in \eqref{def:diam} satisfies
\[
   R_x(t) \leq R_x^0 \qquad\mbox{for all } t\geq 0.
\]
\end{lemma}

The proof of a slight generalization of \cite[Lemma 2.2]{ChoiH1} for the case of variable time delay.
We give it here for the sake of the reader.

%%% Remark: This proof works also for the continuum case

\begin{proof}
Let us fix some $\eps >0$.
We shall prove that for all $t\geq 0$
\(  \label{R-eps}
   R_x(t) < R_x^0 + \eps.
\)
Obviously, $R_x(0) \leq R_x^0$, so that by continuity,
\eqref{R-eps} holds on the maximal interval $[0,T)$ for some $T>0$.
For contradiction, let us assume that $T<+\infty$.
Then we have
\(   \label{R-cont}
   \lim_{t\to T-}  R_x(t) = R_x^0 + \eps. %\qquad\mbox{and}\qquad \tot{}{t+}  R_x(T) \geq 0,
\)
%where $\tot{}{t+}  R_x(T)$ denotes the right-hand side derivative of $R_x$ at $t=T$.
%By continuity, there exists an index $i\in\{1,\dots,N\}$ such that $R_x(t) \equiv |x_i(t)|$
%for $t\in (T,T+\delta)$ for some $\delta>0$.
%on an open right neighborhood of $T$.
However, for any $i=1,\dots,N$, we have
\[
   %\tot{}{t+} R_x(T) =
   \frac12 \tot{}{t} |x_i(t)|^2 &=&  \sum_{j=1}^N \psi_{ij}(t) \left[x_j(t-\tau(t))- x_i(t)\right]\cdot x_i(t) \\
     &=& \sum_{j=1}^N  \psi_{ij}(t) \left[ x_j(t-\tau(t))\cdot x_i(t) - |x_i(t)|^2\right].
\]
%where $\tot{}{t+}$ denotes the derivative with respect to $t$ from the right-hand side.
By definition, we have $|x_j(t-\tau(t))| < R_x^0 + \eps$ for all $j\in\{1,\dots,N\}$ and $t < T$,
so that with \eqref{psi:conv} and an application of the Cauchy-Schwarz inequality we arrive at
\[
   \frac12 \tot{}{t} |x_i(t)|^2 \leq \left(R_x^0 + \eps\right) |x_i(t)| - |x_i(t)|^2.
\]
Now, if $|x_i(t)| \neq 0$, we use the identity $\frac12\tot{|x_i(t)|^2}{t} = |x_i(t)| \tot{|x_i(t)|}{t}$
and divide the above inequality by $|x_i(t)|$.
On the other hand, if $|x_i(t)| \equiv 0$ on an open subinterval of $(0,T)$,
then $\tot{|x_i(t)|}{t} \equiv 0 \leq R^0_x + \eps  - |x_i(t)|$ on this subinterval.
Thus, we obtain
\[
   \tot{}{t} |x_i(t)| \leq R^0_x  + \eps - |x_i(t)| \quad\mbox{a.e. on } (0,T),
\]
which implies
\[
   |x_i(t)| \leq \left(|x_i(0)| - (R^0_x + \eps) \right)e^{-t} + R^0_x + \eps  \qquad\mbox{for } t  < T.
\]
Consequently, with $|x_i(0)| \leq R_x^0$,
\[
   \lim_{t \to T-} \; \max_{1 \leq i \leq N} |x_i(t)| \leq -\eps e^{-T} + R_x^0 + \eps < R^0_x + \eps,
\]
which is a contradiction to \eqref{R-cont}.
We conclude that, indeed, $T = \infty$,
and complete the proof by taking the limit $\eps\to 0$.
\end{proof}

The following geometric result is based on the convexity
property \eqref{psi:conv} of the renormalized communication weights \eqref{psi}.

\begin{lemma}\label{lem:geom}
Let $N\geq 3$ and $\{x_1, \dots, x_N\}\subset \R^d$ be any set of vectors in $\R^d$
and denote $d_x$ its diameter,
$$d_x:=\max_{1 \leq i,j \leq N}|x_i - x_j|.$$
Fix $i, k \in \{1,2,\cdots,N\}$ such that $i\neq k$ and let
$\eta^i_j \geq 0$ for all $j\in\{1,2,\cdots,N\} \setminus\{i\}$, % such that $j\neq i$,
and $\eta^k_j \geq 0$ for all $j\in\{1,2,\cdots,N\} \setminus\{k\}$, % such that $j\neq k$,
such that
\[ % \label{convexity}
   \sum_{j\neq i} \eta^i_j = 1,\qquad \sum_{j\neq k} \eta^k_j = 1.
\]
Let $\mu\geq 0$ be such that
\[ % \label{mu}
   0 \leq \mu \leq \min \left\{ \min_{j\neq i} \eta^i_j,\; \min_{j\neq k} \eta^k_j \right\}. % \geq 0.
\]
Then
\(  \label{est:geom}
   \left| \sum_{j\neq i} \eta^i_j x_j - \sum_{j\neq k} \eta^k_j x_j \right| \leq (1-(N-2)\mu) d_x.
\)
\end{lemma}

\begin{proof}
For $j,\ell \in\{1,2,\cdots,N\} \setminus\{i\}$ we set $\xi^i_{j\ell}:=\mu$ if $\ell\neq j$ and $\xi^i_{jj}:=1 - (N-2)\mu$.
Moreover, define
\(  \label{xji}
   \bar x_j^i := \sum_{\ell\neq i} \xi^i_{j\ell} x_\ell \qquad\mbox{for } j \in \{1,2,\cdots,N\} \setminus\{i\}.
\)
Then the vector $\sum_{j\neq i} \eta^i_j x_j$ can be written as the convex combination of the vectors $\bar x_j^i$,
\(   \label{cc1}
   \sum_{j\neq i} \eta^i_j x_j = \sum_{j\neq i} \lambda_j^i \bar x_j^i
\)
with the coefficients
\[
   \lambda_j^i = \frac{\eta^i_j - \mu}{1-(N-1)\mu} \geq 0, \qquad  \sum_{j\neq i} \eta^i_j = 1.
\]
Similarly, for $j,\ell \in\{1,2,\cdots,N\} \setminus\{k\}$ we set $\xi^k_{j\ell}:=\mu$ if $\ell\neq j$ and $\xi^k_{jj}:=1 - (N-2)\mu$.
Define
\(  \label{xjk}
   \bar x_j^k := \sum_{\ell\neq k} \xi^k_{j\ell} x_\ell \qquad\mbox{for } j \in \{1,2,\cdots,N\} \setminus\{k\},
\)
and again observe that the vector $\sum_{j\neq k} \eta^k_j x_j$ can be written as the convex combination of the vectors $\bar x_j^k$,
\(   \label{cc2}
   \sum_{j\neq k} \eta^k_j x_j = \sum_{j\neq k} \lambda_j^k \bar x_j^k
\)
with the coefficients
\[
   \lambda_j^k = \frac{\eta^k_j - \mu}{1-(N-1)\mu} \geq 0, \qquad  \sum_{j\neq k} \lambda^k_j = 1.
\]
Obviously, due to \eqref{cc1} and \eqref{cc2}, the estimate \eqref{est:geom} is verified as soon as we prove
\[
   \left|\bar x_a^i - \bar x_b^k\right| \leq (1-(N-2)\mu) d_x
\]
for all $a\in\{1,2,\cdots,N\} \setminus\{i\}$ and $b\in\{1,2,\cdots,N\} \setminus\{k\}$.
With \eqref{xji} and \eqref{xjk} we have
\[
     \bar x_a^i - \bar x_b^k &=& \sum_{\ell\neq i} \xi^i_{a\ell} x_\ell - \sum_{\ell\neq k} \xi^k_{b\ell} x_\ell \\
       &=& \mu \sum_{\ell\neq i, a} x_\ell + (1-(N-2)\mu) x_a - \mu \sum_{\ell\neq k, b} x_\ell - (1-(N-2)\mu) x_b \\
       &=& \mu (x_k-x_i) + (1-(N-1)\mu) (x_a-x_b).
\]
With the triangle inequality we then readily obtain
\[
    \left| \bar x_a^i - \bar x_b^k \right| \leq (1-(N-2)\mu) d_x.
\]
\end{proof}

The following Gronwall-Halanay type inequality is a generalization of
\cite[Lemma 2.5]{ChoiH1} for variable time delay.

\begin{lemma}\label{lem:GH}
Let $\tau=\tau(t)$ satisfy the assumptions \eqref{ass:tau}.
Let $u\in C([-\taumax,\infty))$ be a nonnegative continuous function
with piecewise continuous derivative on $(0,\infty)$, such that
for some constant $a \in (0,1)$ the differential inequality is satisfied,
\begin{equation}\label{diff_ineq}
   \tot{}{t}  u(t) \leq (1-a) u(t -\tau(t)) - u(t) \qquad\mbox{for almost all } t>0.
\end{equation}
Then there exists a unique solution $C\in (0,a)$ of the equation
\begin{equation}\label{ass_C}
  1 - C = (1-a)e^{C\taumax}
\end{equation}
and the estimate holds
\( \label{Gronwall-like}
   u(t) \leq \left( \max_{s\in[-\taumax,0]} u(s) \right) e^{-Ct} \qquad \mbox{for all } t \geq 0.
\)
\end{lemma}

\begin{proof}
We denote
\(   \label{def:baru}
    \bar u := \max_{s\in[-\taumax,0]} u(s), \qquad w(t) := \bar u e^{-Ct},
\)
and for any fixed $\lambda > 1$ set
\[
   \mathcal{S}_\lambda := \left\{ t \geq 0 : u(s) \leq \lambda w(s) \quad \mbox{for} \quad s \in [0,t)\right\}.
\]
Since $0 \in \mathcal{S}_\lambda$, $T_\lambda := \sup \mathcal{S}_\lambda \geq 0$ exists. We claim that
\[
   T_\lambda = \infty \qquad \mbox{for any } \lambda >1.
\]
For contradiction, assume $T_\lambda < \infty$ for some $\lambda >1$.
Then clearly $\tau(T_\lambda) > 0$, since otherwise we would have
\[
   \tot{}{t+}  u(t) \leq \lim_{t\to T_\lambda+} (1-a) u(t -\tau(t)) - u(t) = - a u(T_\lambda) = - a \lambda w(T_\lambda) < 0,
\]
which contradicts the definition of $T_\lambda$.
Therefore, due to the continuity of $u=u(t)$ and $\tau=\tau(t)$, there exists some $T_\lambda^* > T_\lambda$ such that
$u$ is differentiable at $T_\lambda^*$ and
\(  \label{est_derivatives}
   u(T_\lambda^*) > \lambda w(T_\lambda^*),\qquad   \tot{}{t} u(T_\lambda^*) > \lambda \tot{}{t} w(T_\lambda^*),\qquad
   \tau(T_\lambda^*) > T_\lambda^*-T_\lambda.
\)
Note that $w$ satisfies
\begin{equation}\label{est_g}
   w(t-\tau(t)) = e^{C\tau(t)}w(t) \quad \mbox{and} \quad \tot{}{t} w(t) = - C w(t),
\end{equation}
for all $t>0$.
Moreover, it follows from \eqref{diff_ineq} that
\begin{equation}\label{est_diff}
   \tot{}{t} u(T_\lambda^*) \leq (1 - a) u(T_\lambda^* - \tau(T_\lambda^*)) - u(T_\lambda^*).
\end{equation}
We now consider the following two cases:
\begin{itemize}
\item
If $T_\lambda^* - \tau(T_\lambda^*) \leq 0$, then by \eqref{def:baru} we have $u(T_\lambda^* - \tau(T_\lambda^*)) \leq \bar u$
and $u(T_\lambda^*) > \lambda w(T_\lambda^*)$ by \eqref{est_derivatives},
so that we estimate the right-hand side of \eqref{est_diff} by
$$
\begin{aligned}
    \tot{}{t} u(T_\lambda^*)  &\leq (1-a) \bar u - u(T_\lambda^*)\cr
   &< (1-a) \lambda w(0) - \lambda w(T_\lambda^*)\cr
   &\leq \left((1 - a)e^{C\taumax} - 1\right) \lambda w(T_\lambda^*).
\end{aligned}
$$
For the third line we used the inequality $w(0) \leq w(T_\lambda^*)e^{C\taumax}$
implied by $T_\lambda^* \leq \tau(T_\lambda^*) \leq \taumax$.
With the identities \eqref{ass_C} and \eqref{est_g}, we obtain
$$
   \tot{}{t} u(T_\lambda^*)  < - C \lambda w(T_\lambda^*)  = \lambda \tot{}{t} w(T_\lambda^*),
$$
which is a contradiction to \eqref{est_derivatives}.

\item
If $T_\lambda^* - \tau(T_\lambda^*) > 0$, we have $u(T_\lambda^* - \tau(T_\lambda^*)) \leq \lambda w(T_\lambda^* - \tau(T_\lambda^*))$
since, due to \eqref{est_derivatives}, $T_\lambda^* - \tau(T_\lambda^*) < T_\lambda$.
Then \eqref{est_diff} gives
\[
   \tot{}{t} u(T_\lambda^*) &\leq& (1-a) \lambda w(T_\lambda^* - \tau(T_\lambda^*)) - \lambda w(T_\lambda^*)\\
     &=& \left((1 - a)e^{C\tau(T_\lambda^*)} - 1\right) \lambda w(T_\lambda^*) \\
     &\leq& \left((1 - a)e^{C\taumax} - 1\right) \lambda w(T_\lambda^*).
\]
Then, using the same argument as in the previous case, we obtain a contradiction to \eqref{est_derivatives}.
\end{itemize}
We conclude that, for every $\lambda>1$, $T_\lambda = \infty$ and $u(t) \leq \lambda w(t)$ for all $t\geq 0$.
Passing to the limit $\lambda \to 1$ yields the claim \eqref{Gronwall-like}.
\end{proof}

We are now prepared to prove Theorem \ref{thm:HK}.

\begin{proof}
For the sake of legibility, let us introduce the shorthand notation $\widetilde x_j := x_j(t-\tau(t))$,
while $x_j$ means $x_j(t)$.

The uniform bound on the radius $R_x=R_x(t)$ of the solution provided by Lemma \ref{lem:Rxbound} gives
for all $i,j\in \{1,2,\cdots,N\}$,
\[
   |\widetilde x_j - x_i| \leq |\widetilde x_j| + |\widetilde x_i| \leq 2R_x^0 \qquad\mbox{for all } t\geq 0.
\]
Consequently, defining
\(  \label{def:upsi}
    \underline{\psi} := \min_{s\in[0,2R_x^0]} \psi(s),
\)
we have $\psi(|\widetilde x_j - x_i|) \geq \underline{\psi}$, and
\(  \label{psibound}
   \psi_{ij} = \frac{\psi(|\widetilde x_j - x_i|)}{\sum_{\ell\neq i} \psi(|\widetilde x_\ell - x_i|)} \geq \frac{\underline{\psi}}{N-1}.
\)
Note that due to the assumption \eqref{ass:psi}, we have $\underline{\psi} >0$.

Due to the continuity of the trajectories $x_i=x_i(t)$,
there is an at most countable system of open, mutually disjoint
intervals $\{\mathcal{I}_\sigma\}_{\sigma\in\N}$ such that
$$
   \bigcup_{\sigma\in\N} \overline{\mathcal{I}_\sigma} = [0,\infty)
$$
and for each ${\sigma\in\N}$ there exist indices $i(\sigma)$, $k(\sigma)$
such that
$$
   d_x(t) = |x_{i(\sigma)}(t) - x_{k(\sigma)}(t)| \quad\mbox{for } t\in \mathcal{I}_\sigma.
$$
Then, using the abbreviated notation $i:=i(\sigma)$, $k:=k(\sigma)$,
we have for every $t\in \mathcal{I}_\sigma$,
\[
   \frac12 \tot{}{t} d_x(t)^2 &=& (\dot x_i - \dot x_k)\cdot (x_i-x_k)\\
      &=& 
       \left(\sum_{j\neq i} \psi_{ij} (\widetilde x_j - x_i) - \sum_{j\neq k} \psi_{kj} (\widetilde x_j - x_k) \right) \cdot (x_i-x_k)\\
      &=&
      \left( \sum_{j\neq i} \psi_{ij} \widetilde x_j  - \sum_{j\neq k} \psi_{kj} \widetilde x_j \right) \cdot (x_i-x_k) - |x_i-x_k|^2,
\]
where we used the convexity property of the renormalized weights \eqref{psi:conv}.
We now use \eqref{psibound} and Lemma \ref{lem:geom} with $\mu:=\frac{\underline{\psi}}{N-1}$, which gives
\[
   \left| \sum_{j\neq i} \psi_{ij} \widetilde x_j  - \sum_{j\neq k} \psi_{kj} \widetilde x_j \right| \leq (1-(N-2)\mu) d_x(t-\tau(t)).
\]
Consequently, with the Cauchy-Schwartz inequality we have
\[
   \frac12 \tot{}{t} d_x(t)^2 \leq (1-(N-2)\mu) d_x(t-\tau) d_x(t) - d_x(t)^2,
\]
which implies that for almost all $t>0$,
\[
   \tot{}{t} d_x(t) \leq (1-(N-2)\mu) d_x(t-\tau(t)) - d_x(t).
\]
An application of Lemma \ref{lem:GH} with $a:=(N-2)\mu = \frac{N-2}{N-1}\underline{\psi}  \in (0,1)$ gives then
the exponential decay
\[
   d_x(t) \leq \left( \max_{s\in[-\taumax,0]} d_x(s) \right) e^{-Ct} \qquad \mbox{for } t \geq 0,
\]
where $C$ is the unique solution of \eqref{ass_C}.
We note that $a$ increases with increasing $N$ (if $\underline{\psi}$ is held constant), and so does $C$. %, the solution of \eqref{ass_C}.
Consequently, the exponential decay rate $C$ improves with increasing $N$.
\end{proof}

%%%%%%%%%%%%%%%%%%%%%%%%%%%%%%%%%%
\section{Asymptotic flocking for the Cucker-Smale model}\label{sec:CS}
The method develop in Section \ref{sec:HK} can be easily extended for
the Cucker-Smale model \eqref{eq:CS}, as we demonstrate in the proof of Theorem \ref{thm:CS} below.

\begin{proof}
First, note that the proof of Lemma \ref{lem:Rxbound} applies
mutatis mutandis for the velocity variable in \eqref{eq:CS}, providing
the uniform bound
\( \label{Rvbound}
   R_v(t) := \max_{1\leq i\leq N} |v_i(t)| \leq R_v^0 \qquad\mbox{for all } t\geq 0,
\)
with $R_v^0 := \max_{t\in [-\taumax,0]} R_v(t)$.

Let $C\in (0,1)$ be given by \eqref{ass:C}.
With $d_v^0>0$ given by \eqref{Rdd0} and due to the continuity of $d_v=d_v(t)$, there exists some $T>0$ such that
\(   \label{CSbound}
   \int_0^t d_v(s) \d s < \frac{d_v^0}{C} \qquad \mbox{for all } t<T.
\)
We claim that $T=\infty$. For contradiction, assume that \eqref{CSbound} holds only until some finite $T>0$.
Then we have
\(  \label{forContr}
   \int_0^T d_v(s) \d s = \frac{d_v^0}{C}. 
\)
By the first equation of \eqref{eq:CS} we readily have
\[
   d_x(t) \leq d_x^0 + \int_0^t d_v(s) \d s,
\]
so that \eqref{CSbound} implies for all $t<T$,
\[
   d_x(t) \leq d_x^0 + \int_0^t d_v(s) \d s < d_x^0 + \frac{d_v^0}{C}.
\]
Moreover, using the estimate
\[
   |\widetilde x_j - x_j| = \left| \int_{t-\tau(t)}^t \dot x_j(s) \d s \right| \leq \int_{t-\tau(t)}^t |v_j(s)| \d s \leq \taumax R_v^0,
\]
provided by \eqref{Rvbound}, we have for any $i, j \in \{1,2,\cdots,N\}$, $i\neq j$,
\[
   |\widetilde x_j - x_i| \leq |\widetilde x_j - x_j| + |x_j-x_i| \leq \taumax R_v^0 + d_x(t) \leq \taumax R_v^0 + d_x^0 + \frac{d_v^0}{C}.
\]
Then by the definition \eqref{def:Psi} of $\Psi$,
\[
   \psi(|\widetilde x_j - x_i|) \geq \Psi\left( \taumax R_v^0 + d_x^0 + \frac{d_v^0}{C} \right),  % \qquad\mbox{for all } t<T.
\]
and by the universal bound $\psi\leq 1$,
\(  \label{psiboundCS}
   \psi_{ij} = \frac{\psi(|\widetilde x_j - x_i|)}{\sum_{\ell\neq i} \psi(|\widetilde x_\ell - x_i|)} \geq \frac{1}{N-1} \Psi\left( \taumax R_v^0 + d_x^0 + \frac{d_v^0}{C} \right).
\)
Similarly as in the proof of Theorem \ref{thm:HK}, we have for $t<T$ such that $d_v=|v_i-v_k|$ on some neighborhood of $t$,
\(   \nonumber
   \frac12 \tot{}{t} d_v(t)^2 &=& \left( \sum_{j\neq i} \psi_{ij} \widetilde v_j  - \sum_{j\neq k} \psi_{kj} \widetilde v_j \right) \cdot (v_i-v_k) - |v_i-v_k|^2 \\
     &\leq& \left| \sum_{j\neq i} \psi_{ij} \widetilde v_j  - \sum_{j\neq k} \psi_{kj} \widetilde v_j \right| d_v(t) - d_v(t)^2.
     \label{ddtdv}
\)
We now use \eqref{psiboundCS} and Lemma \ref{lem:geom} with $\mu:=\frac{1}{N-1}\Psi\left( \taumax R_v^0 + d_x^0 + \frac{d_v^0}{C} \right)$, which gives
\[
   \left| \sum_{j\neq i} \psi_{ij} \widetilde v_j  - \sum_{j\neq k} \psi_{kj} \widetilde v_j \right| \leq (1-(N-2)\mu) d_v(t-\tau(t)).
\]
Consequently, \eqref{ddtdv} implies that for almost all $t>0$,
%\[ \frac12 \tot{}{t} d_v(t)^2 \leq (1-(N-2)\mu) d_v(t-\tau(t)) d_v(t) - d_v(t)^2, \]
\[
   \tot{}{t} d_v(t) \leq (1-(N-2)\mu) d_v(t-\tau(t)) - d_v(t).
\]
An application of Lemma \ref{lem:GH} with
\[
   a := (N-2)\mu = \frac{N-2}{N-1} \psi\left( \taumax R_v^0 + d_x^0 + \frac{d_v^0}{C} \right),
\]
recalling \eqref{ass:C}, gives
\(  \label{decay:dv}
   d_v(t) \leq d_v^0 e^{-Ct} \qquad\mbox{for all } t<T.
\)
But then
\[
   \int_0^T d_v(s) \d s \leq d_v^0 \int_0^T  e^{-Cs} \d s = \frac{d_v^0}{C} \left( 1 - e^{-CT} \right) < \frac{d_v^0}{C},
\]
which is a contradiction to \eqref{forContr}.
Thus, we conclude that $T=\infty$, i.e., that \eqref{CSbound} holds for all $t>0$.
Then also \eqref{decay:dv} holds for all $t>0$, and, moreover,
$d_x=d_x(t)$ is uniformly bounded by $d_x^0 + \frac{d_v^0}{C}$.
\end{proof}

Finally, we provide the proof of Corollary \ref{corr:CS}. Considering the monotone communication rate function
\[
   \psi(s) = \frac{1}{(1+s^2)^\beta}
\]
with $\beta\geq 0$, we obviously have $\Psi(r) = \psi(r)$ for all $r\geq 0$, with $\Psi(r)$ defined in \eqref{def:Psi}.
Then it is straightforward to calculate
\[
   \lim_{C\to 0+} \left( 1- \frac{N-2}{N-1}\Psi\left(\taumax R_v^0 + d_x^0 + \frac{d_v^0}{C} \right)\right) e^{C\taumax} = 1,
\]
for any positive values of $R_v^0$, $d_x^0$, $d_v^0$ and $\taumax$.
On the other hand, the above expression is strictly positive for $C=1$.
Therefore, \eqref{ass:C} is solvable with some $C\in(0,1)$ as soon as
\[
   %\lim_{C\to 0+}
   \tot{}{C} \left[ \left( 1- \frac{N-2}{N-1}\psi\left(\taumax R_v^0 + d_x^0 + \frac{d_v^0}{C} \right)\right) e^{C\taumax} \right]_{C=0+} < -1.
\]
A simple calculation reveals that this is the case if $\beta<1/2$.

%%%%%%%%%%%%%%%%%%%%%%%%%%%%%%%%%%
\section{Consensus and flocking in the mean-field limit}\label{sec:MF}
Our results for the mean-field limit systems \eqref{mf:HK}--\eqref{F} and \eqref{mf:CS}--\eqref{G}
are based on the well-posedness theory in measures developed in \cite[Section 3]{ChoiH1}.
In particular, existence and uniqueness of measure-valued solutions for the Cucker-Smale
system \eqref{mf:CS}--\eqref{G} was proved there, together with continuous dependence
on the initial datum. The proof uses the framework developed in \cite{CCR}
and is based on local Lipschitz continuity of the operator $F=F[f]$ given by \eqref{F}.
Without going into details, we note that the proof can be easily adapted
to provide analogous results for the Hegselmann-Krause system \eqref{mf:HK}--\eqref{F}.
Instead, we merely restate the stability results in Wasserstein distance for the two systems,
which are essential for our proof of asymptotic consensus and flocking.
For their proof we refer to \cite[Theorem 3.6]{ChoiH1}.

\begin{theorem}\label{thm:stabHK}
Let $f_1, f_2 \in C([0, T];\P(\R^{d}))$ be two measure-valued solutions of \eqref{mf:HK}--\eqref{F} on the time interval $[0, T]$,
subject to the compactly supported initial data $f_1^0, f_2^0 \in C([-\taumax, 0];\P_1(\R^{d}))$.
Then there exists a constant $L=L(T)$ such that
\[ %  \label{d-estimate}
   \W_1 (f_1(t),f_2(t)) \leq L \max_{s\in[-\taumax,0]} \W_1(f^0_1(s),f^0_2(s)) \quad \mbox{for} \quad t \in [0,T],
\]
where $\W_1(f_1(t),f_2(t))$ denotes the 1-Wasserstein (or Monge-Kantorovich-Rubinstein) distance \cite{Villani}
of the probability measures $f_1(t)$, $f_2(t)$. %, see \cite{Definition 3.1]{ChoiH1}.
\end{theorem}

\begin{theorem}\label{thm:stabCS}
Let $g_1, g_2 \in C([0, T];\P(\R^{d}\times\R^{d}))$ be two measure-valued solutions of \eqref{mf:CS}--\eqref{G} on the time interval $[0, T]$,
subject to the compactly supported initial data $g_1^0, g_2^0 \in C([-\taumax, 0];\P_1(\R^{d}\times\R^{d}))$.
Then there exists a constant $L=L(T)$ such that
\[ % \label{d-estimate}
   \W_1 (g_1(t),g_2(t)) \leq L \max_{s\in[-\taumax,0]} \W_1(g^0_1(s),g^0_2(s)) \quad \mbox{for} \quad t \in [0,T].
\]
\end{theorem}

We are now in position to provide a proof of Theorem \ref{thm:mf:HK}.

\begin{proof}
Fixing an initial datum $f^0 \in C([-\taumax,0], \P(\R^d))$, uniformly compactly supported in the sense of \eqref{IC:mf:HK},
we construct $\{f^0_N\}_{N\in\N}$ a family of $N$-particle approximations of $f^0$, i.e.,
\[
   f^0_N(s,x) := \frac{1}{N} \sum_{i=1}^N  \delta(x-x^0_i(s)) \qquad\mbox{for } s\in[-\taumax,0],
\]
where the $x_i^0\in C([-\taumax,0];\R^d)$ are chosen such that
\[
   \max_{s\in[-\taumax,0]} \W_1(f^0_N(s),f^0(s)) \to 0 \quad\mbox{as}\quad N\to\infty.
\]
Denoting then $x^N_i=x^N_i(t)$ the solution of the discrete Hegselmann-Krause system \eqref{eq:HK}
subject to the initial datum $x_i^0=x_i^0(s)$, $i=1,\dots,N$, the proof of Theorem \ref{thm:HK}
gives exponential convergence to global consensus, i.e.,
\[
   d_x(t) \leq \left( \max_{s\in[-\taumax,0]} d_x(s) \right) e^{-C_N t} \qquad \mbox{for } t \geq 0,
\]
with the diameter $d_x$ defined in \eqref{def:dX}, $C_N$ the unique solution of \eqref{ass_C}
with $a_N := \frac{N-2}{N-1}\underline{\psi}  \in (0,1)$ and $\underline{\psi}$ given by \eqref{def:upsi}.
Note that $C_N$ increases with $N$ and $\lim_{N\to\infty} C_N = C$,
with $C$ the unique solution of \eqref{ass_C} with $a:=\underline{\psi}$.
%To obtain a convergence rate independent of $N$, observe that $C_N \geq C_3$ for all $N\geq 3$.
%for some constant $C\in(0,1)$ independent of $N$, with the diameter $d_x$ defined in \eqref{def:dX}.
It is easy to check that the empirical measure
\[
   f^N(t,x) := \frac{1}{N} \sum_{i=1}^N  \delta(x-x^N_i(t))
\]
is a measure valued solution of \eqref{eq:HK}.
For any fixed $T>0$, Theorem \ref{thm:stabHK} provides the stability estimate
\[
   \W_1 (f(t),f^N(t)) \leq L \max_{s\in[-\taumax,0]} \W_1 (f^0(s),f^0_N(s)) \qquad \mbox{for} \quad t \in [0,T),
\]
where the constant $L>0$ is independent of $N$.
Thus, fixing $T>0$ and letting $N\to\infty$ implies $d_x[f(t)] = d_x(t)$ on $[0,T)$, and, consequently,
\[
   d_x[f(t)] \leq \left( \max_{s\in[-\taumax,0]} d_x[f^0(s)] \right) e^{-Ct} \qquad \mbox{for } t \in[0,T).
\]
We conclude by noting that $T>0$ can be chosen arbitrarily and that the constant $C$ is independent of time.
\end{proof}

The proof of Theorem \ref{thm:mf:CS} is an obvious modification of the above proof,
with stability provided by Theorem \ref{thm:stabCS}.

%%%%%%%%%%%%%%%%%%%%%%%%%%%%%%%%%%
\section*{Acknowledgment}
JH acknowledges the support of the KAUST baseline funds.

%SIAM J. Applied Dyn Syst, J. Differential Equations

%%%%%%%%%%%%%%%%%%%%%%%%%%%%%%%%%%

\end{document}